\documentclass[11pt]{amsart}

\usepackage{booktabs,amsmath}
\usepackage{amssymb}
\usepackage{amsthm}
\usepackage{graphics}
\usepackage{tikz}
\usetikzlibrary{arrows,automata,positioning}
\usepackage{tikz-qtree}
\usepackage{marvosym}
\usetikzlibrary{decorations.pathreplacing}
\usepackage{amsfonts}
\usepackage{mathtools}
\usepackage{algorithm,algpseudocode}
\usepackage{mathdots}
\usepackage{float}
\usepackage{caption}
\usepackage{subcaption}
\usepackage{stmaryrd}
\usepackage{mathrsfs}
\usepackage{soul}
\allowdisplaybreaks

\usepackage{color} 
\definecolor{ao(english)}{rgb}{0.0, 0.5, 0.0}

\newtheorem{lemma}{Lemma}
\newtheorem{theorem}{Theorem}
\newtheorem{proposition}{Proposition}
\newtheorem{corollary}{Corollary}

\begin{document}
\title{Solving the membership problem for certain subgroups of $SL_2(\mathbb Z)$}

\author{Sandie Han, Ariane M. Masuda, Satyanand Singh, and Johann Thiel}
\date{\today}
\address{Department of Mathematics, New York City College of Technology, The City University of New York (CUNY), 300 Jay Street,
Brooklyn, New York 11201}
\email{\{shu.han51,ariane.masuda70,satyanand.singh03,johann.thiel30\}@citytech.cuny.edu}
\thanks{The second author received support for this project provided by a PSC-CUNY award, jointly funded by The Professional Staff Congress and The City University of New York. }

\begin{abstract}
For positive integers $u$ and $v$, let $L_u=\begin{bmatrix}1 & 0 \\u&1\end{bmatrix}$ and $R_v=\begin{bmatrix}1 & v \\ 0 & 1\end{bmatrix}$. Let $G_{u,v}$ be the group generated by $L_u$ and $R_v$. In a previous paper, the authors determined a characterization of matrices $M=\begin{bmatrix}a & c \\b&d\end{bmatrix}$ in $G_{u,v}$ when $u,v\geq 3$ in terms of the short continued fraction representation of $b/d$. We extend this result to the case where $u+v> 4$. Additionally, we compute $[\mathscr{G}_{u,v}\colon G_{u,v}]$ for $u,v\geq 1$, extending a result of Chorna, Geller, and Shpilrain.
\end{abstract}
\maketitle

\section{Introduction}

For positive integers $u$ and $v$, let $L_u=\begin{bmatrix} 1 & 0 \\ u & 1 \end{bmatrix}$, $R_v=\begin{bmatrix} 1 & v \\ 0 & 1 \end{bmatrix}$, and $G_{u,v}$ be the group generated by $L_u$ and $R_v$. Using the notation from~\cite{EG1}, let
$$\mathscr{G}_{u,v} = \left\{\begin{bmatrix}1+uvn_1 & vn_2\\ un_3 & 1+uvn_4\end{bmatrix}\in SL_2(\mathbb{Z})\colon (n_1,n_2,n_3,n_4)\in\mathbb{Z}^4\right\}.$$ 
Note that $\mathscr{G}_{u,v}$ is a group and that $G_{u,v}\subseteq\mathscr{G}_{u,v}$\footnote{The case when $u,v\geq 2$ is handled in~\cite[Proposition 1.1]{HMST}, but the more general version stated here follows using the same argument.}.

Given a rational number $q$, if there exist integers $q_0,q_1,\dots,q_r$ (referred to as partial quotients) such that $$q = q_0+\cfrac{1}{q_1+\cfrac{1}{q_2+\ddots+\cfrac{1}{q_r }}},$$
then we refer to such an identity as a continued fraction representation of $q$ and denote it by $[q_0,q_1,\dots,q_r].$ We refer to the unique such representation where $q_i\geq 1$ for $0<i<r$ and $q_r>1$ for $r>0$ as {\em the} short continued fraction representation of $q$.

In~\cite{EG1}, Esbelin and Gutan gave the following clear characterization of members of $G_{k,k}$ when $k\geq 3$ in terms of related continued fraction representations.

\begin{theorem}[Esbelin and Gutan~\cite{EG1}]\label{egGp}
For an integer $k\geq 2$ and a matrix $M=\begin{bmatrix}a & b \\c & d\end{bmatrix}\in\mathscr{G}_{k,k}$,  $M\in G_{k,k}$ if and only if at least one of the rationals $c/a$ and $b/d$ has a continued fraction expansion having all partial quotients in $k\mathbb{Z}.$
\end{theorem}

In~\cite{HMST}, we showed that Theorem~\ref{egGp} could be modified and written in terms of the short continued fraction representations of either $c/a$ or $b/d$, when $u,v\geq 3$. In particular, we developed a simple algorithm that, when applied to the short continued fraction representation of $b/d$, determines whether or not the sought after continued fraction expansion in Theorem~\ref{egGp} exists. Before stating our result, we need some additional definitions.

Let $A=\bigcup_{r=0}^\infty (\mathbb{Z}\times\mathbb{Z}_{\neq 0}^r)$. We denote an element of $A$ by $\llbracket q_0,q_1,\dots,q_r\rrbracket.$ Let $$-\llbracket q_0,q_1,\dots,q_r\rrbracket :=\llbracket -q_0,-q_1,\dots,-q_r\rrbracket.$$ For any nonnegative integers $m$ and $n$, let $$\llbracket  q_0,q_1,\dots,q_m\rrbracket\oplus\llbracket  p_0,p_1,\dots,p_n\rrbracket := \begin{cases}
\llbracket q_0,q_1,\dots,q_m,p_0,p_1,\dots,p_n\rrbracket & \text{ if }p_0\neq 0,\\
\llbracket q_0,q_1,\dots,q_m+p_1,p_2,\dots,p_n\rrbracket &\text{ otherwise.}
\end{cases}$$ 
Let 
\begin{align*}
    A_0 &= \{\llbracket q_0,q_1,\dots,q_r\rrbracket\in A\colon  [q_i,\dots,q_r]\neq 0\text{ when } 0<i<r\},\\
    A_1 &=\{\llbracket q_0,q_1,\dots,q_r\rrbracket\in A_0\colon q_i\geq 1\text{ when } 0<i<r,\text{ and } q_r>1\text{ when } r>0\}, \text{ and}\\
    A_2 &=\{\llbracket q_0,q_1,\dots,q_r\rrbracket\in A_0\colon |q_i|> 1 \text{ when } 0<i\leq r\}.
\end{align*} 
Define the function $C\colon\mathbb{Q}\to A_1$ by 
$$C(x)=\llbracket x_0,x_1,\dots,x_r\rrbracket$$ if $[x_0,x_1,\dots,x_r]$ is the short continued fraction representation of $x$. Define a function $f\colon A_1\to A_2$ recursively by
\begin{center}
\begin{align*}
&f(\llbracket q_0,q_1,\dots,q_r\rrbracket) \\
=&\begin{cases}
\llbracket q_0,q_1,\dots,q_r\rrbracket  &\text{ if $r=0$ or $q_i\neq 1$ for $0<i< r$,}\\
\llbracket q_0,q_1,\dots,q_{j-1}+1\rrbracket\oplus-f(\llbracket q_{j+1}+1,q_{j+2},\dots,q_r\rrbracket) &\text{ if $q_j=1$ and $q_i\neq 1$ for $0< i < j$.}
\end{cases}
\end{align*}
\end{center}
Lastly, we say that $\llbracket q_0,q_1,\dots,q_r\rrbracket\in A$ satisfies the $(u,v)$-divisibility property if $v|q_i$ when $i$ is even, and $u|q_i$ when $i$ is odd.

 We are now able to state the result from~\cite{HMST} that we intend to extend.

\begin{theorem}[Han et al.~\cite{HMST}]\label{sanovlikeGpOld}
For integers $u,v\geq 3$ and a matrix $M=\begin{bmatrix}
    a & b \\
    c & d
    \end{bmatrix}\in \mathscr{G}_{u,v}$, $M\in G_{u,v}$ if and only if $(f\circ C)(b/d)$ satisfies the $(u,v)$-divisibility property.
\end{theorem}

A careful review of the function $f$ shows that it seeks to eliminate 1's appearing in the short continued fraction representation of $b/d$. When $u,v\geq 3$, if $M\in G_{u,v}$, any such 1's in the short continued fraction representation of $b/d$ encode where the exponents in the alternating product representation\footnote{Note that $G_{u,v}$ is freely generated when $u,v\geq 2$ (see~\cite{N1}).} of $M$ in terms of $L_u$ and $R_v$ change sign. The price paid for eliminating 1's in this way is that the adjacent partial quotients are modified using Lemma~\ref{firstid} below. The reason that $f$ cannot be applied in the case where either $u=2$ or $v=2$ is that the continued fraction representation of $b/d$ may contain consecutive 1's, not all of which correspond to sign changes. Haphazardly removing the first 1 encountered, as $f$ does, can lead to incorrect results. 

We must therefore introduce an alternative that is sensitive to the existence \emph{and} location of 1's in the continued fraction representation of $b/d$. This necessitates the introduction of a new family of functions designed to supersede $f$ in the more general setting of $u+v>4.$ Our motivation for this new family of functions is based on the following general principle. Let $M=\begin{bmatrix}a & b \\c & d\end{bmatrix}$ and $b/d = [q_0,q_1,\dots,q_r]$. Multiplication by $L_u$ and $R_v$ changes the continued fraction representation of $b/d$ in a predictable way (see Lemma~\ref{cf}). Therefore, any matrix $M\in G_{u,v}$ must have entries $b$ and $d$ such that $b/d$ has a continued fraction representation with very specific properties. Any such deviation would immediately lead one to conclude that $M$ is not in the group. The main point is that the continued fraction representation of $b/d$ given by the division algorithm is closely connected to the one obtained by tracing the product of $L_u$'s and $R_v$'s. Our goal is to find all of the (finite) possible ways in which the former can be transformed into the latter.

In~\cite{S}, Sanov shows that $G_{2,2}=\mathscr{G}_{2,2}$ and in~\cite{CGS}, Chorna, Geller, and Shpilrain show that $[\mathscr{G}_{k,k}\colon G_{k,k}]=\infty$ for $k\geq 3$. In this paper we  will show that for $u,v\geq 1$,
\begin{align*}
[\mathscr{G}_{u,v}\colon G_{u,v}]=
\begin{cases}
1 & u+v\leq 4,\\
\infty &\text{otherwise,}
\end{cases}
\end{align*}
by extending the techniques developed for the proof of Theorem~\ref{sanovlikeGpOld} that also led to an alternate proof of Sanov's result in~\cite{HMST}.

\section{Results}

We begin with some useful identities whose proofs can be shown algebraically.

\begin{lemma}\label{firstid}
Suppose that $\alpha,\beta,\gamma\in\mathbb{Z}$ are such that all quantities below are well-defined. Then 
\begin{enumerate}
    \item[(a)] $[\alpha,1,\beta,\gamma]=[\alpha+1,-(\beta+1),-\gamma],$
    \item[(b)] $[\alpha,\beta,\gamma]=[\alpha+1,-1,-(\beta-1),-\gamma],\text{ and}$
    \item[(c)] $[\alpha,0,\gamma]=[\alpha+\gamma].$ 
\end{enumerate}
\end{lemma}

\begin{lemma}\label{secondid}
Suppose that $\alpha,\beta,\gamma\in\mathbb{Z}$ with $\beta>0$ and $\gamma\neq -1,0$. Then 
\begin{align*}
[\alpha,\beta,\gamma] &=[\alpha, \underbrace{1,0,1,0,\dots,0,1}_{2\beta-1\text{ terms}},\gamma]\\
&= [\alpha+1,\underbrace{-2,2,-2,\dots,\pm 2}_{\beta-1\text{ terms}},(-1)^\beta(\gamma+1)].
\end{align*}
\end{lemma}

Note that Lemma~\ref{secondid} represents a correction of Lemma 3.5 in~\cite{HMST} where the $(-1)^\beta$ term was mistakenly given as $-1.$ The correction does not invalidate the alternate proof of Sanov's result given in~\cite{HMST}, however the previous incorrect version can give product representations for matrices in $G_{2,2}$ whose exponents are off by a sign.

By Theorem 1 in~\cite{CJR} it is clear that all subsequent results that hold with $v\geq u$ also hold with $u > v$. We will assume that $v\geq u$ moving forward unless stated otherwise.

\subsection{A New Family of Functions}

Before we introduce a new family of functions to supersede $f$, we must expand some of our earlier definitions. This is done to avoid having to complicate matters later by having to consider a large number of special cases.

Let $A^\ast = A\cup \{\lambda\}$ where $\lambda$ is a formal symbol playing the role of, say, an empty vector. That is, we regard $\lambda$ to be shorthand for $\llbracket\phantom{.}\rrbracket.$ Similarly, we define $A_i^\ast = A_i\cup\{\lambda\}$ for $i=0,1,2$. We also say that $-\lambda = \lambda$, and for any element $\llbracket q_0,q_1,\dots,q_r\rrbracket$ in $A^\ast$, let
\begin{align*}
    \lambda \oplus \llbracket q_0,q_1,\dots,q_r\rrbracket &= \llbracket q_0,q_1,\dots,q_r\rrbracket,\\
    \llbracket q_0,q_1,\dots,q_r\rrbracket \oplus \lambda &= \llbracket q_0,q_1,\dots,q_r\rrbracket,\text{ and}\\
    \lambda \oplus \lambda &= \lambda.
\end{align*}

Let $A_3 =\{\llbracket q_0,q_1,\dots,q_r\rrbracket\in A_0\colon q_i\neq 0 \text{ when } 0<i\leq r\}$ and $A_3^\ast = A_3 \cup \{\lambda\}.$
For $u+v>4$, define $f_{u,v}\colon A_1^\ast\to A_3^\ast$ recursively by $f_{u,v}(\lambda) = \lambda$ and

\begin{center}
\begin{align*}
&f_{u,v}(\llbracket q_0,q_1,\dots,q_r\rrbracket) \\
=&\begin{cases}
\llbracket q_0\mp 1,\pm 1\rrbracket &\text{ if $r=0$, $q_0\equiv \pm 1\bmod{v}$, and $u=1$,} \\
\llbracket q_0+1\rrbracket\oplus -f_{v,u}(\llbracket q_2+1,q_3,\dots,q_r\rrbracket) &\text{ if $q_0\equiv -1\bmod{v}$ and $q_1=1$,} \\
\llbracket q_0- 1, 1\rrbracket\oplus -
f_{u,v}(\llbracket q_1+ 1,q_2,\dots,q_r\rrbracket) &\text{ if $u=1$, $q_0\equiv 1\bmod{v}$, and $q_1\geq 1$}, \\
\llbracket q_0+ 1,- 1\rrbracket\oplus -
f_{u,v}(\llbracket q_1- 1,q_2,\dots,q_r\rrbracket) &\text{ if $u=1$, $q_0\equiv -1\bmod{v}$, and $q_1> 1$},\\
\llbracket q_0+2,-1\rrbracket\oplus f_{u,v}(\llbracket q_2+2,q_3,\dots,q_r\rrbracket) &\text{ if $u=1$, $q_0\equiv -2\bmod{v}$, and $q_1=1$,} \\
\llbracket q_0+2,-1\rrbracket\oplus f_{u,v}(\llbracket 2, 1, q_2,\dots,q_r\rrbracket) &\text{ if $u=1$, $q_0\equiv -2\bmod{v}$, and $q_1=2$,} \\
\llbracket q_0+1\rrbracket\oplus - f_{v,u}(\llbracket 1, 1, q_2,\dots,q_r\rrbracket) &\text{ if $u=2$, $q_0\equiv -1\bmod{v}$, and $q_1=2$}, \\
\llbracket q_0\rrbracket \oplus f_{v,u}(\llbracket q_1,\dots,q_r\rrbracket) &\text{ otherwise.}
\end{cases}
\end{align*}
\end{center}

\noindent Define $g\colon A_3^\ast\to A_3^\ast$ recursively by $g(\lambda) = \lambda$ and 
\begin{align*}
g(\llbracket q_0,q_1,q_2,\dots,q_r\rrbracket)
=&\begin{cases}
\llbracket q_0\pm1\rrbracket&\text{ if $r=1$ and $q_1=\pm 1$,}\\
\llbracket q_0-1,2\rrbracket&\text{ if $r=1$ and $q_1=-2$,}\\
\llbracket q_0-1,1\rrbracket \oplus g(-\llbracket q_1+1,q_2,\dots,q_r\rrbracket)&\text{ if $q_1<-2$ or both $r>1$ and $q_1=-2$,}\\
g(\llbracket q_0-1\rrbracket \oplus -\llbracket q_2-1,q_3,\dots,q_r\rrbracket)&\text{ if $r>1$ and $q_1=-1$},\\
\llbracket q_0\rrbracket \oplus g(\llbracket q_1,\dots,q_r\rrbracket)&\text{ otherwise.}
\end{cases}
\end{align*}
Let $E:A_0\to\mathbb{Q}$ be given by $$E(\llbracket q_0,q_1,\dots,q_r\rrbracket) = q_0+\cfrac{1}{q_1+\cfrac{1}{q_2+\ddots+\cfrac{1}{q_r }}}.$$

\begin{lemma}[{\cite[Lemma 2.3]{HMST}}]\label{Eid}
For $a,b\in A_0$ with $a=\llbracket q_0,q_1,\dots,q_r\rrbracket$ and $b=\llbracket p_0,p_1,\dots,p_s\rrbracket$, 
$$E(a\oplus b) = 
q_0+\cfrac{1}{q_1+\ddots+\cfrac{1}{q_r+\cfrac{1}{E(b)}}}
$$
or equivalently,
$$E(a\oplus b) = 
\begin{cases}
[ q_0,q_1,\dots,q_r,p_0,p_1,\dots,p_s ] &\text{ if } p_0\neq 0,\\
[ q_0,q_1,\dots,q_r+p_1,p_2,\dots,p_s ] &\text{ if } p_0=0.
\end{cases}
$$
Also, $E(-a)=-E(a)$.
\end{lemma}

Note that Lemma~\ref{Eid} is a correction of~\cite[Lemma 2.3]{HMST}. Namely, the original statement gave the wrong formula in the case where $p_0=0$. However, this case went unused in that paper and therefore did not introduce any errors. In the results that follow we will make use of the $p_0=0$ case.

The following corollaries follow from Lemmas~\ref{firstid},~\ref{secondid}, and~\ref{Eid}. In particular, we show that the definitions for $f_{u,v}$ and $g$ are such that $E\circ f_{u,v}\circ C$ and $E\circ g\circ C$ act as the identity map.

\begin{corollary}\label{efuv}
For all $\llbracket q_0,q_1,\dots,q_r\rrbracket\in A_1$, $E(\llbracket q_0,q_1,\dots,q_r\rrbracket) = (E\circ f_{u,v})(\llbracket q_0,q_1,\dots,q_r\rrbracket).$
\end{corollary}

\begin{proof}
We will prove the desired result by strong induction on $r$. 

Suppose that $r=0$. Then the result follows immediately for all possible pairs of $u$ and $v$ from the definitions of $E$ and $f_{u,v}$ since  $f_{u,v}(\llbracket q_0\rrbracket)$ is equal to either $\llbracket q_0\rrbracket$ or $\llbracket q_0\mp1,\pm1\rrbracket$. In either case we have $q_0=(E\circ f_{u,v})(\llbracket q_0\rrbracket)$.

Suppose that the result holds for $0\leq r\leq t$ for some $t\geq 0$ and all possible pairs of $u$ and $v$.

\noindent{\bf Case 1:} Suppose $q_0\equiv -1\bmod{v}$ and $q_1=1$. Then
\begin{align*}
    &(E\circ f_{u,v})(\llbracket q_0,q_1,\dots,q_{t+1}\rrbracket)\\
    =& E(\llbracket q_0+1\rrbracket\oplus -f_{v,u}(\llbracket q_2+1,q_3,\dots,q_{t+1}\rrbracket))\text{ by the definition of $f_{u,v}$}\\
    =& q_0+1 + \frac{1}{-(E\circ f_{v,u})(\llbracket q_2+1,q_3,\dots,q_{t+1}\rrbracket)}\text{ by Lemma~\ref{Eid}}\\
    =& [q_0+1, -(q_2+1),-q_3,\dots,-q_{t+1}]\text{ by the induction hypothesis}\\
    =& [q_0,1,q_2,q_3,\dots,q_{t+1}]\text{ by Lemma~\ref{firstid}(a)}\\
    =& E(\llbracket q_0,q_1,\dots,q_{t+1}\rrbracket).
\end{align*}

\noindent{\bf Case 2:} Suppose $u=1$, $q_0\equiv 1\bmod{v}$, and $q_1\geq 1$. Then
\begin{align*}
    &(E\circ f_{u,v})(\llbracket q_0,q_1,\dots,q_{t+1}\rrbracket)\\
    =& E(\llbracket q_0- 1, 1\rrbracket\oplus -f_{u,v}(\llbracket q_1+ 1,q_2,\dots,q_{t+1}\rrbracket))\text{ by the definition of } f_{u,v}\\
    =& q_0- 1+\cfrac{1}{1+\cfrac{1}{-(E\circ f_{u,v})(\llbracket q_1+1,q_2,\dots,q_{t+1}\rrbracket)}}\text{ by Lemma~\ref{Eid}}\\
    =&[q_0- 1, 1, -(q_1+ 1), -q_2, -q_3,\dots, -q_{t+1}] \text{ by the induction hypothesis} \\
     =&[q_0,q_1, q_2,q_3,\dots,q_{t+1}]\text{ by Lemma~\ref{firstid}(a)} \\
    =& E(\llbracket q_0,q_1,\dots,q_{t+1}\rrbracket).
\end{align*}

\noindent{\bf Case 3:} Suppose $u=1$, $q_0\equiv -1\bmod{v}$, and $q_1> 1$. Then
\begin{align*}
    &(E\circ f_{u,v})(\llbracket q_0,q_1,\dots,q_{t+1}\rrbracket)\\
    =& E(\llbracket q_0+1, -1\rrbracket\oplus -f_{u,v}(\llbracket q_1- 1,q_2,\dots,q_{t+1}\rrbracket))\text{ by the definition of } f_{u,v}\\
    =& q_0+1+\cfrac{1}{-1+\cfrac{1}{-(E\circ f_{u,v})(\llbracket q_1-1,q_2,\dots,q_{t+1}\rrbracket)}}\text{ by Lemma~\ref{Eid}}\\
    =&[q_0+1, -1, -(q_1- 1), -q_2, -q_3,\dots, -q_{t+1}] \text{ by the induction hypothesis} \\
     =&[q_0,q_1, q_2,q_3,\dots,q_{t+1}]\text{ by Lemma~\ref{firstid}(b)} \\
    =& E(\llbracket q_0,q_1,\dots,q_{t+1}\rrbracket).
\end{align*}

\noindent{\bf Case 4:} Suppose $u=1$, $q_0\equiv -2\bmod{v}$, and $q_1=1$. Then
\begin{align*}
    &(E\circ f_{u,v})(\llbracket q_0,q_1,\dots,q_{t+1}\rrbracket)\\
    =& E(\llbracket q_0+2,-1\rrbracket\oplus f_{u,v}(\llbracket q_2+2,q_3,\dots,q_{t+1}\rrbracket)\text{ by the definition of $f_{u,v}$}\\
    =& q_0+2+\cfrac{1}{-1+\cfrac{1}{(E\circ f_{u,v})(\llbracket q_2+2,q_3,\dots,q_{t+1}\rrbracket)}}\text{ by Lemma~\ref{Eid}}\\
    =& [q_0+2, -1, q_2+2, q_3,\dots, q_{t+1}]\text{ by the induction hypothesis}\\
    =& [q_0+1,-(q_2+1),-q_3,\dots,-q_{t+1}]\text{ by Lemma~\ref{firstid}(b)}\\
    =& [q_0, 1,q_2,q_3,\dots,q_{t+1}]\text{ by Lemma~\ref{firstid}(a)}\\
    =& E(\llbracket q_0,q_1,\dots,q_{t+1}\rrbracket).
\end{align*}

\noindent{\bf Case 5:}
Suppose $u=1$, $q_0\equiv -2\bmod v$, and $q_1=2$. Then
\begin{align*}
    &(E\circ f_{u,v})(\llbracket q_0,q_1,\dots,q_{t+1}\rrbracket)\\
    =& E(\llbracket q_0+2,-1\rrbracket\oplus f_{u,v}(\llbracket 2,1,q_2,q_3,\dots,q_{t+1}\rrbracket))\text{ by the definition of } f_{u,v}\\
    =& q_0+2+\cfrac{1}{-1+\cfrac{1}{2+\cfrac{1}{1+\cfrac{1}{(E\circ f_{u,v})(\llbracket q_2,q_3,\dots,q_{t+1}\rrbracket)}}}}\text{ by Lemma~\ref{Eid}}\\
    =&[q_0+2,-1, 2, 1, q_2,q_3,\dots,q_{t+1}] \text{ by the induction hypothesis} \\
     =&[q_0+1,-1, -1, -q_2,-q_3,\dots,-q_{t+1}]\text{ by Lemma~\ref{firstid}(b)} \\
     =&[q_0,2,  q_2,q_3,\dots,q_{t+1}]\text{ by Lemma~\ref{firstid}(b)} \\
    =& E(\llbracket q_0,q_1,\dots,q_{t+1}\rrbracket).
\end{align*}

 \noindent{\bf Case 6:} Suppose $u=2$, $q_0\equiv -1 \bmod v$, and $q_1=2$. Then
\begin{align*}
    &(E\circ f_{u,v})(\llbracket q_0,q_1,\dots,q_{t+1}\rrbracket)\\
    =& E(\llbracket q_0+1\rrbracket\oplus -f_{v,u}(\llbracket 1,1,q_2,q_3,\dots,q_{t+1}\rrbracket)\text{ by the definition of $f_{u,v}$}\\
    =& q_0+1 + \dfrac{1}{-(E\circ f_{v,u})(\llbracket 1,1,q_2,q_3,\dots,q_{t+1}\rrbracket)}\text{ by Lemma~\ref{Eid}}\\
    =& [ q_0+1, -1,-1,-q_2,-q_3,\dots,-q_{t+1}]\text{ by the induction hypothesis}\\
    =& [ q_0,2,q_2,q_3,\dots,q_{t+1}]\text{ by Lemma~\ref{firstid}(b)}\\
    =& E(\llbracket q_0,q_1,\dots,q_{t+1}\rrbracket).
\end{align*}

The remaining case is a triviality. Having exhausted all possibilities, the result follows by induction.
\end{proof}

\begin{corollary}\label{eguv}
For all $\llbracket q_0,q_1,\dots,q_r\rrbracket\in A_3$, $E(\llbracket q_0,q_1,\dots,q_r\rrbracket) = (E\circ g)(\llbracket q_0,q_1,\dots,q_r\rrbracket).$
\end{corollary}
\begin{proof}
We will prove the desired result by strong induction on $r$.

Suppose that $r = 0$. Then the result follows immediately since $g(\llbracket q_0\rrbracket) = q_0$.

Suppose that the result holds for $0\leq r\leq t$ for some $t\geq 0$.

\noindent{\bf Case 1:} Suppose $q_1 < -1$. Then
\begin{align*}
&(E\circ g)(\llbracket q_0,q_1,\dots,q_{t+1}\rrbracket)\\
= & E(\llbracket q_0-1,1\rrbracket \oplus g(-\llbracket q_1+1,q_2,\dots,q_{t+1}\rrbracket) \text{ by the definition of $g$}\\
=& q_0-1+\cfrac{1}{1+\cfrac{1}{(E\circ g)(-\llbracket q_1+1,q_2,\dots,q_{t+1}\rrbracket)}}  \text{ by Lemma~\ref{Eid}}\\
=& [q_0-1,1,-(q_1+1),-q_2\dots,-q_{t+1}] \text{ by the induction hypothesis}\\
=& [q_0,q_1,\dots,q_{t+1}] \text{ by Lemma~\ref{firstid}(a)}\\
=& E(\llbracket q_0,q_1,\dots,q_{t+1}\rrbracket).
\end{align*}

\noindent{\bf Case 2:} Suppose $q_1 = -1$ and $q_2=1$. Then
\begin{align*}
&(E\circ g)(\llbracket q_0,q_1,\dots,q_{t+1}\rrbracket)\\
= & (E\circ g)(\llbracket q_0-1\rrbracket \oplus-\llbracket q_2-1,q_3,\dots,q_{t+1}\rrbracket) \text{ by the definition of $g$}\\
=& (E\circ g)(\llbracket q_0-q_3-1,-q_4,\dots,-q_{t+1}\rrbracket) \text{ by the definition of $\oplus$} \\
=&[q_0-q_3-1,-q_4,\dots,-q_{t+1}] \text{ by the induction hypothesis}\\
=&  [q_0-1,0,-q_3,-q_4,\dots,-q_{t+1}]  \text{ by Lemma~\ref{firstid}(c)} \\
=& [q_0,-1,1,q_3,\dots,q_{t+1}] \text{ by Lemma~\ref{firstid}(b)}\\
=&  E(\llbracket q_0,q_1,\dots,q_{t+1}\rrbracket).
\end{align*}

\noindent{\bf Case 3:} Suppose $q_1 = -1$ and $q_2\neq 1$. Then
\begin{align*}
&(E\circ g)(\llbracket q_0,q_1,\dots,q_{t+1}\rrbracket)\\
= & (E\circ g)(\llbracket q_0-1\rrbracket \oplus-\llbracket q_2-1,q_3,\dots,q_{t+1}\rrbracket) \text{ by the definition of $g$}\\
=& (E\circ g)(\llbracket q_0-1,-(q_2-1),-q_3,\dots,-q_{t+1}\rrbracket) \text{ by the definition of $\oplus$} \\ 
=& [q_0-1,-(q_2-1),-q_3,\dots,-q_{t+1}] \text{ by the induction hypothesis} \\
=& [q_0,-1,q_2,q_3,\dots,q_{t+1}]  \text{ by Lemma~\ref{firstid}(b)}  \\
=&  E(\llbracket q_0,q_1,\dots,q_{t+1}\rrbracket).
\end{align*}

For the remaining case, $q_1>0$, if $r=1$, then the result is trivial and if $r>1$, then it is similar to case 1. Having exhausted all possibilities, the result follows by induction.
\end{proof}

We have shown thus far that $f_{u,v}$ and $g$ are well-behaved with respect to continued fraction representations, much like the original $f$ and $g$ were in~\cite{HMST}. The functions $f_{u,v}$ attempt to find equivalent continued fraction representations satisfying the $(u,v)$-divisibility property while $g$ computes the short continued fraction representation or a rational number. 

Before we can proceed to the main results, we need to show that the new definition of $g$ continues to hold similar properties to those of the function of the same name as it was originally defined in~\cite{HMST}.

\begin{lemma}\label{firstEntry}
For all $\llbracket q_0,q_1,\dots,q_r\rrbracket\in A_3$, if $$g(\llbracket q_0,q_1,\dots,q_r\rrbracket) = \llbracket q'_0,q'_1,\dots,q'_s\rrbracket,$$ then for any integer $n$, $$g(\llbracket q_0+n,q_1,\dots,q_r\rrbracket) = \llbracket q'_0+n,q'_1,\dots,q'_s\rrbracket.$$
\end{lemma}
The proof of Lemma~\ref{firstEntry} is trivial and therefore omitted.

\begin{lemma}\label{nonNegFirst}
For all $\llbracket q_0,q_1,\dots,q_r\rrbracket\in A_3$, if $$g(\llbracket q_0,q_1,\dots,q_r\rrbracket)=\llbracket q'_0,q'_1,\dots,q'_s\rrbracket$$ is such that $\llbracket q_0,q_1,\dots,q_r\rrbracket$ satisfies the $(u,v)$-divisibility property with $u+v>4$ and $q_0>0$, then $q'_0\geq 0$. In particular, $q'_0$ is equal to either $q_0+1$, $q_0$, $q_0-1$ or $q_0-2$, with the latter only possible when $u=1$ and $v\geq 4$.
\end{lemma}

\begin{proof}
We can prove the lemma by examining three cases. Note that by the definition of $g$, we can assume that $r>1$.

\noindent{\bf Case 1:} Suppose $q_1>-1$. Then $q'_0=q_0\geq 0$.

\noindent{\bf Case 2:} Suppose $q_1<-1$. Then $q'_0=q_0-1\geq 0$.

\noindent{\bf Case 3:} Suppose $q_1=-1$. Then $ q_0\geq 4$, so
\begin{align*}
    g(\llbracket q_0,q_1,q_2,\dots,q_r\rrbracket) &= g(\llbracket q_0-1\rrbracket\oplus-\llbracket q_2-1,\dots,q_r\rrbracket).
\end{align*}
Since it cannot be the case that $q_2=2$, it follows that this case resolves to a previous case with $q_0-1$ in place of $q_0$. From this we see that the desired result holds.

Having exhausted all possible cases, the statement of the lemma holds.
\end{proof}

\begin{proposition}\label{gToA1}
If $\llbracket q_0,q_1,\dots,q_r\rrbracket\in A_3$ satisfies the $(u,v)$-divisibility property with $u+v>4$, then 
$g(\llbracket q_0,q_1,\dots,q_r\rrbracket)\in A_1$.
\end{proposition}

\begin{proof}
We will prove the desired result by strong induction on $r$. Note that by Lemma~\ref{firstEntry}, we may assume that $q_0$ is a positive multiple of $v$.

Suppose that $r=0$. Then the result follows immediately since $g(\llbracket q_0\rrbracket)=q_0$.

Suppose that $r=1$. If $q_1=\pm 1$, then $g(\llbracket q_0,\pm1\rrbracket) = \llbracket q_0\pm 1\rrbracket\in A_1$. If $q_1=-2$, then $g(\llbracket q_0,-2\rrbracket) = \llbracket q_0-1,2\rrbracket\in A_1$. If $q_1<-2$, then $g(\llbracket q_0,q_1\rrbracket) = \llbracket q_0-1,1,-(q_1+1)\rrbracket\in A_1$. If $q_1>1$, then $g(\llbracket q_0,q_1\rrbracket) = \llbracket q_0,q_1\rrbracket\in A_1$. 

Suppose that the result holds for $0\leq r\leq t$ for some $t\geq 1$ and all pairs $u$ and $v$ with $u+v>4$.

\noindent{\bf Case 1:} Suppose $q_1> -1$. Then 
\begin{align*}
     g(\llbracket q_0,q_1,q_2,\dots,q_{t+1}\rrbracket) &=  \llbracket q_0\rrbracket \oplus g(\llbracket q_1,\dots,q_{t+1}\rrbracket)\\
     &= \llbracket q_0\rrbracket \oplus \llbracket p_0,\dots,p_s\rrbracket.
\end{align*}
In this case we have that $\llbracket q_1,\dots,q_{t+1}\rrbracket$ satisfies the $(v,u)$-divisibility property, so $\llbracket p_0,\dots,p_s\rrbracket\in A_1$ by the induction hypothesis and $p_0\geq 0$ by Lemma~\ref{nonNegFirst}. In particular, $p_i>0$ for $0<i\leq s$. If $p_0>0$, then
\begin{equation*}
    \llbracket q_0\rrbracket\oplus \llbracket p_0,\dots,p_s\rrbracket = \llbracket q_0,p_0,\dots,p_s\rrbracket.
\end{equation*}
Whereas if $p_0=0$, then
\begin{equation*}
    \llbracket q_0\rrbracket\oplus \llbracket p_0,\dots,p_s\rrbracket = \llbracket q_0+p_1,\dots,p_s\rrbracket.
\end{equation*}
In either situation we get the desired result.

\noindent{\bf Case 2:} Suppose $q_1<-1$. Then
\begin{align*}
    g(\llbracket q_0,q_1,q_2,\dots,q_{t+1}\rrbracket) &= \llbracket q_0-1,1\rrbracket\oplus g(-\llbracket q_1-1,q_2,\dots,q_{t+1}\rrbracket)\\
    &= \llbracket q_0-1,1\rrbracket\oplus g(-\llbracket q_1-1,q_2,\dots,q_{t+1}\rrbracket)\\
    &= \llbracket q_0-1,1\rrbracket\oplus \llbracket p_0+1,p_1,\dots,p_s\rrbracket
\end{align*}
 by Lemma~\ref{firstEntry}, where
\begin{equation*}
    g(-\llbracket q_1,q_2,\dots,q_{t+1}\rrbracket) = \llbracket p_0,p_1,\dots,p_s\rrbracket.
\end{equation*}
Since $-\llbracket q_1,q_2,\dots,q_{t+1}\rrbracket$ satisfies the $(v,u)$-divisibility property and $-q_1>1$, then $\llbracket p_0,p_1,\dots,p_s\rrbracket\in A_1$ by the induction hypothesis, so $\llbracket p_0+1,p_1\dots,p_s\rrbracket\in A_1$. The remaining portion of this case follows similarly to case 1.

\noindent{\bf Case 3:} Suppose $q_1=-1$. Then
\begin{align*}
    g(\llbracket q_0,q_1,q_2,\dots,q_{t+1}\rrbracket) &= g(\llbracket q_0-1\rrbracket\oplus-\llbracket q_2-1,\dots,q_{t+1}\rrbracket)\\
    &= g(\llbracket q_0-1,1-q_2,-q_3,\dots,-q_{t+1}\rrbracket).
\end{align*}
If $q_2<0$, then 
\begin{equation*}
    g(\llbracket q_0,q_1,q_2,\dots,q_{t+1}\rrbracket) = \llbracket q_0-1\rrbracket \oplus g(\llbracket 1-q_2,-q_3,\dots,-q_{t+1}\rrbracket)
\end{equation*}
and we can use an argument similar to case 2. Alternatively, if $q_2>0$, then
\begin{equation*}
    g(\llbracket q_0,q_1,q_2,\dots,q_{t+1}\rrbracket) = \llbracket q_0-2,1\rrbracket \oplus g(\llbracket q_2-2,q_3,\dots,q_{t+1}\rrbracket)
\end{equation*}
and we can, again, use an argument similar to case 2 since it must be the case that $q_2\geq 4$.

Having exhausted all possibilities, the result follows by strong induction.
\end{proof}

\begin{corollary}\label{guvE}
If $\llbracket q_0,q_1,\dots,q_r\rrbracket\in A_3$ satisfies the $(u,v)$-divisibility property with $u+v>4$, then $$g(\llbracket q_0,q_1,\dots,q_r\rrbracket)=(C\circ E)(\llbracket q_0,q_1,\dots,q_r\rrbracket).$$
\end{corollary}

\begin{proof}
    Let $\llbracket q'_0,q'_1,\dots,q'_s\rrbracket=g(\llbracket q_0,q_1,\dots,q_r\rrbracket)$. By Proposition~\ref{gToA1}, $\llbracket q'_0,q'_1,\dots,q'_s\rrbracket\in A_1$. Furthermore, $q'_s>1$, so $[q'_0,q'_1,\dots,q'_s]$ is the short continued fraction representation of a rational number. By Corollary~\ref{eguv} and the uniqueness of short continued fraction representations, it must be the case that 
\begin{align*}
    (C\circ E)(\llbracket q_0,q_1,\dots,q_r\rrbracket) &= (C\circ E\circ g)(\llbracket q_0,q_1,\dots,q_r\rrbracket)\\
    &= (C\circ E)(\llbracket q'_0,q'_1,\dots,q'_s\rrbracket)\\
    &= C([q'_0,q'_1,\dots,q'_s])\\
    &=\llbracket q'_0,q'_1,\dots,q'_s\rrbracket\\
    &=g(\llbracket q_0,q_1,\dots,q_r\rrbracket).
\end{align*}
\end{proof}

We are now in a position to prove the analogue of Proposition 2.9 in~\cite{HMST}, albeit under a more restricted hypothesis.

\begin{proposition}\label{fuvg}
If $\llbracket q_0,q_1,\dots,q_r\rrbracket\in A_3$ satisfies the $(u,v)$-divisibility property with $u+v>4$, then 
\begin{align*}
(f_{u,v}\circ g)(\llbracket q_0,q_1,\dots,q_r\rrbracket) = \llbracket q_0,q_1,\dots,q_r\rrbracket.\label{uvuv}
\end{align*} 
\end{proposition}

\begin{proof}
We will prove the desired result by strong induction on $r$.

Suppose that $r=0$. Then the result follows immediately for all possible pairs $u$ and $v$ since $(f_{u,v}\circ g)(\llbracket q_0\rrbracket) = f_{u,v}(\llbracket q_0\rrbracket)=\llbracket q_0\rrbracket$. 

Suppose that $r=1$. If $q_1=\pm 1$, then 
\begin{align*}
    (f_{u,v}\circ g)(\llbracket q_0,\pm 1\rrbracket) &= f_{u,v}(\llbracket q_0\pm 1\rrbracket) \\
    &= \llbracket q_0,\pm 1\rrbracket. 
\end{align*}
If $q_1=-2$, then it must be the case that $u=1$ or $2$. If $u=1$, then
\begin{align*}
    (f_{1,v}\circ g)(\llbracket q_0,-2\rrbracket) &= f_{1,v}(\llbracket q_0-1,2\rrbracket) \\
    &= \llbracket q_0,-1\rrbracket \oplus -f_{1,v}(\llbracket 1\rrbracket) \\
    &= \llbracket q_0,-1\rrbracket \oplus -\llbracket 0,1\rrbracket \\
    &= \llbracket q_0,-2\rrbracket.
\end{align*}
Otherwise,
\begin{align*}
    (f_{2,v}\circ g)(\llbracket q_0,-2\rrbracket)&= f_{2,v}(\llbracket q_0-1,2\rrbracket) \\
    &= \llbracket q_0\rrbracket \oplus - f_{v,2}(\llbracket 1,1\rrbracket) \\
    &= \llbracket q_0\rrbracket \oplus -(\llbracket 2\rrbracket \oplus - f_{2,v}(\lambda)) \\
    &= \llbracket q_0\rrbracket \oplus -(\llbracket 2\rrbracket \oplus  \lambda)\\
    &= \llbracket q_0,-2\rrbracket.
\end{align*}
If $q_1<-2$, then 
\begin{align*}
    (f_{u,v}\circ g)(\llbracket q_0,q_1\rrbracket) &= f_{u,v}(\llbracket q_0-1,1\rrbracket \oplus g(-\llbracket q_1+1\rrbracket)) \\
    &= f_{u,v}(\llbracket q_0-1,1,-(q_1+1)\rrbracket) \\
    &= \llbracket q_0\rrbracket \oplus -f_{v,u}(\llbracket -q_1\rrbracket) \\
    &= \llbracket q_0,q_1\rrbracket 
\end{align*}
If $q_1>1$, then
\begin{align*}
    (f_{u,v}\circ g)(\llbracket q_0,q_1\rrbracket) &= f_{u,v}(\llbracket q_0\rrbracket \oplus g(\llbracket q_1\rrbracket)) \\
    &= f_{u,v}(\llbracket q_0,q_1\rrbracket)\\
    &= \llbracket q_0\rrbracket \oplus f_{v,u}(\llbracket q_1\rrbracket) \\
    &= \llbracket q_0,q_1\rrbracket. 
\end{align*}
Having exhausted all possibilities, we have that the result holds for $r=1$.

Suppose that the result holds for $0\leq r\leq t$ for some $t\geq 1$ and all pairs $u$ and $v$ with $u+v>4$.

\noindent{\bf Case 1:} Suppose $q_1>0$. Then 
\begin{align*}
    (f_{u,v}\circ g)(\llbracket q_0,q_1,q_2,\dots,q_{t+1}\rrbracket) &= f_{u,v}(\llbracket q_0\rrbracket\oplus g(\llbracket q_1,q_2,\dots,q_{t+1}\rrbracket))\\
    &= \llbracket q_0\rrbracket\oplus (f_{v,u}\circ g)(\llbracket q_1,q_2,\dots,q_{t+1}\rrbracket)\\
    &= \llbracket q_0,q_1,q_2,\dots,q_{t+1}\rrbracket {\text{ by the induction hypothesis,}}
\end{align*}
as desired.

\noindent{\bf Case 2:} Suppose $q_1<-1$. Let
\begin{equation}
\label{q1.less.neg1}
    \llbracket q'_0,\dots,q'_s\rrbracket = g(-\llbracket q_1+1,q_2,\dots,q_{t+1}\rrbracket).
\end{equation}
Then 
\begin{align*}
    (f_{u,v}\circ g)(\llbracket q_0,q_1,q_2,\dots,q_{t+1}\rrbracket) &= f_{u,v}(\llbracket q_0-1,1\rrbracket \oplus g(-\llbracket q_1+1,q_2,\dots,q_{t+1}\rrbracket)) \\
    &= f_{u,v}(\llbracket q_0-1,1\rrbracket \oplus \llbracket q'_0,\dots,q'_s\rrbracket),
\end{align*}
where $q'_0\geq 0$ by Lemma~\ref{nonNegFirst}. If $q'_0>0$, then, using Lemma~\ref{firstEntry} and the induction hypothesis,
\begin{align*}
    (f_{u,v}\circ g)(\llbracket q_0,q_1,q_2,\dots,q_{t+1}\rrbracket) &= f_{u,v}(\llbracket q_0-1,1,q'_0,\dots,q'_s\rrbracket)\\
    &= \llbracket q_0\rrbracket \oplus -f_{v,u}(\llbracket q'_0+1,q'_1,\dots,q'_s\rrbracket) \\
    &= \llbracket q_0\rrbracket \oplus -(f_{v,u}\circ g)(-\llbracket q_1,q_2,\dots,q_{t+1}\rrbracket)  {\text{ by Lemma~\ref{firstEntry}}}\\
    &= \llbracket q_0\rrbracket \oplus \llbracket q_1,q_2,\dots,q_{t+1}\rrbracket\\
    &= \llbracket q_0,q_1,q_2,\dots,q_{t+1}\rrbracket.
\end{align*}
Otherwise, if $q'_0=0$, then using~\eqref{q1.less.neg1} and Lemma~\ref{firstEntry},
\begin{equation}
\label{eq:qprime0.is.0}
g(\llbracket -q_1,-q_2,\dots,-q_{t+1}\rrbracket) = \llbracket 1, q'_1,q'_2,\dots,q'_s\rrbracket.
\end{equation}
In view of Lemma~\ref{nonNegFirst}, for~\eqref{eq:qprime0.is.0} to hold, it must be the case that $u\leq 2$, $q_1=-2$, and $q_2\geq3$. In particular,
\begin{align}
\label{eq:q1.neg.2}
    (f_{u,v}\circ g)(\llbracket q_0,q_1,q_2,\dots,q_{t+1}\rrbracket) &= (f_{u,v}\circ g)(\llbracket q_0,-2,q_2,\dots,q_{t+1}\rrbracket)\\
    &= f_{u,v}(\llbracket q_0-1,1\rrbracket\oplus g(-\llbracket-1,q_2,\dots,q_{t+1}\rrbracket))\notag\\
    &= f_{u,v}(\llbracket q_0-1,1\rrbracket\oplus (\llbracket 0,1\rrbracket\oplus g(\llbracket q_2-1,q_3,\dots,q_{t+1}\rrbracket)))\notag\\
    &= f_{u,v}(\llbracket q_0-1,2\rrbracket\oplus g(\llbracket q_2-1,q_3,\dots,q_{t+1}\rrbracket)).\notag
\end{align}
Let
\begin{equation}
\label{q2.pos}
    \llbracket q^\ast_0,q^\ast_1,\dots,q^\ast_w\rrbracket = g(\llbracket q_2,q_3,\dots,q_{t+1}\rrbracket).
\end{equation}
If $u=1$, then using~\eqref{eq:q1.neg.2},~\eqref{q2.pos}, and Lemma~\ref{firstEntry},
\begin{align}
\label{eq:fuv.q2}
    (f_{u,v}\circ g)(\llbracket q_0,q_1,q_2,\dots,q_{t+1}\rrbracket) &= \llbracket q_0,-1\rrbracket\oplus-f_{u,v}(\llbracket 1\rrbracket\oplus g(\llbracket q_2-1,q_3,\dots,q_{t+1}\rrbracket)) \\
    &= \llbracket q_0,-1\rrbracket\oplus-f_{u,v}(\llbracket 1\rrbracket\oplus \llbracket q^\ast_0-1,q^\ast_1,\dots,q^\ast_w\rrbracket) {\text{ by Lemma~\ref{firstEntry}}}.
\end{align}
Note that since $u=1$, we must have $q_2\geq 4$, so by Lemma~\ref{nonNegFirst}, $q^\ast_0\geq 2$. It follows using~\eqref{q2.pos} and~\eqref{eq:fuv.q2} that
\begin{align*}
    (f_{u,v}\circ g)(\llbracket q_0,q_1,q_2,\dots,q_{t+1}\rrbracket) &= \llbracket q_0,-1\rrbracket\oplus -f_{u,v}(\llbracket 1, q^\ast_0-1,q^\ast_1,\dots,q^\ast_w\rrbracket)\\
    &= \llbracket q_0,-1\rrbracket\oplus -(\llbracket 0,1\rrbracket\oplus -f_{u,v}(\llbracket q^\ast_0,q^\ast_1,\dots,q^\ast_w\rrbracket)) \\
    &= \llbracket q_0,-2\rrbracket\oplus (f_{u,v}\circ g)(\llbracket q_2,q_3,\dots,q_{t+1}\rrbracket)\\
    &= \llbracket q_0,-2,q_2,q_3,\dots,q_{t+1}\rrbracket {\text{ by the induction hypothesis.}}
\end{align*}
Alternatively, suppose that $u=2$. It again follows from Lemma~\ref{nonNegFirst} that $q^\ast_0\geq 2$. Using~\eqref{eq:q1.neg.2},~\eqref{q2.pos}, and Lemma~\ref{firstEntry},
\begin{align*}
    (f_{u,v}\circ g)(\llbracket q_0,q_1,q_2,\dots,q_{t+1}\rrbracket) &= \llbracket q_0\rrbracket\oplus - f_{v,u}(\llbracket 1,1\rrbracket\oplus g(\llbracket q_2-1,q_3,\dots,q_{t+1}\rrbracket)) \\
    &= \llbracket q_0\rrbracket\oplus - f_{v,u}(\llbracket 1,1\rrbracket\oplus \llbracket q^\ast_0-1,q^\ast_1,\dots,q^\ast_w\rrbracket)\\
    &= \llbracket q_0\rrbracket\oplus-(\llbracket 2\rrbracket\oplus-f_{u,v}(\llbracket q^\ast_0,q^\ast_1,\dots,q^\ast_w\rrbracket)) \\
    &= \llbracket q_0\rrbracket\oplus-(\llbracket 2\rrbracket\oplus-(f_{u,v}\circ g)(\llbracket q_2,q_3,\dots,q_{t+1}\rrbracket))\\
    &= \llbracket q_0,-2,q_2,q_3,\dots,q_{t+1}\rrbracket  {\text{ by the induction hypothesis.}}
\end{align*}
So the result holds in this case.

\noindent{\bf Case 3:} Suppose $q_1=-1$. Let
\begin{equation}
\label{q2.neg}
    \llbracket q^\dag_0,q^\dag_1,\dots,q^\dag_y\rrbracket = g(\llbracket -q_2,-q_3,\dots,-q_{t+1}\rrbracket).
\end{equation}
By assumption, we must have $u=1$ and $v\geq 4$. If $q_2\leq -4$, then, using~\eqref{q2.neg} and Lemma~\ref{firstEntry},
\begin{align*}
    (f_{u,v}\circ g)(\llbracket q_0,-1,q_2,\dots,q_{t+1}\rrbracket) &= (f_{u,v}\circ g)(\llbracket q_0-1,1-q_2,-q_3,\dots,-q_{t+1}\rrbracket) \\
    &= f_{u,v}(\llbracket q_0-1\rrbracket\oplus g(\llbracket 1-q_2,-q_3,\dots,-q_{t+1}\rrbracket)) \\
    &= f_{u,v}(\llbracket q_0-1\rrbracket\oplus \llbracket q^\dag_0+1,q^\dag_1,\dots,q^\dag_y\rrbracket) {\text{ by Lemma~\ref{firstEntry}}}\\
    &= f_{u,v}(\llbracket q_0-1,q^\dag_0+1,q^\dag_1,\dots,q^\dag_y\rrbracket).
\end{align*}
It follows that $q^\dag_0\geq 2$ by Lemma~\ref{nonNegFirst}. Now using~\eqref{q2.neg}, 
\begin{align*}
    (f_{u,v}\circ g)(\llbracket q_0,-1,q_2,\dots,q_{t+1}\rrbracket) &= \llbracket q_0,-1\rrbracket\oplus -f_{u,v}(\llbracket q^\dag_0,q^\dag_1,\dots,q^\dag_y\rrbracket)\\
    &= \llbracket q_0,-1\rrbracket\oplus -(f_{u,v}\circ g)(\llbracket -q_2,-q_3,\dots,-q_{t+1}\rrbracket)\\
    &= \llbracket q_0,-1,q_2,\dots,q_{t+1}\rrbracket {\text{ by the induction hypothesis.}}
\end{align*}
If, instead, $q_2\geq 4$, then, using~\eqref{q2.pos} and Lemma~\ref{firstEntry},
\begin{align}
\label{q_double_prime_start}
    (f_{u,v}\circ g)(\llbracket q_0,-1,q_2,\dots,q_{t+1}\rrbracket) &= (f_{u,v}\circ g)(\llbracket q_0-1,1-q_2,-q_3,\dots,-q_{t+1}\rrbracket) \\
    &= f_{u,v}(\llbracket q_0-2,1\rrbracket\oplus g(-\llbracket 2-q_2,-q_3,\dots,-q_{t+1}\rrbracket)) \\
    &= f_{u,v}(\llbracket q_0-2,1\rrbracket\oplus\llbracket q^\ast_0-2,q^\ast_1,\dots,q^\ast_w\rrbracket)\notag.
\end{align}
Note that $q^\ast_0\geq 2$ by Lemma~\ref{nonNegFirst}. If $q^\ast_0\geq 3$, then using~\eqref{q2.pos} and~\eqref{q_double_prime_start},
\begin{align*}
    (f_{u,v}\circ g)(\llbracket q_0,-1,q_2,\dots,q_{t+1}\rrbracket) &= f_{u,v}(\llbracket q_0-2,1,q^\ast_0-2,q^\ast_1,\dots,q^\ast_w\rrbracket)\\
    &= \llbracket q_0,-1\rrbracket\oplus f_{u,v}(\llbracket q^\ast_0,q^\ast_1,\dots,q^\ast_w\rrbracket) \\
    &= \llbracket q_0,-1\rrbracket\oplus (f_{u,v}\circ g)(\llbracket q_2,q_3,\dots,q_{t+1}\rrbracket)\\
    &= \llbracket q_0,-1,q_2,\dots,q_{t+1}\rrbracket {\text{ by the induction hypothesis.}}
\end{align*}
Alternatively, if $q^\ast_0=2$, then by Lemma~\ref{nonNegFirst}, we must have that $t\geq 3$ with $v=4$, $q_2=4$, $q_3=-1$, and $q_4\geq 4$. Let
\begin{equation}
\label{q4}
    \llbracket q^\diamond_0, q^\diamond_1,\dots,q^\diamond_z\rrbracket = g(\llbracket q_4,q_5,\dots,q_{t+1}\rrbracket).
\end{equation}
Now, using~\eqref{q4} and Lemma~\ref{firstEntry},
\begin{align*}
    (f_{u,v}\circ g)(\llbracket q_0,-1,4,-1,q_4,\dots,q_{t+1}\rrbracket) &= (f_{u,v}\circ g)(\llbracket q_0-1\rrbracket\oplus -\llbracket 3,-1,q_4,\dots,q_{t+1}\rrbracket)\\
    &= (f_{u,v}\circ g)(\llbracket q_0-1, -3,1,-q_4,\dots,-q_{t+1}\rrbracket)\\
    &= f_{u,v}(\llbracket q_0-2,1\rrbracket\oplus g(-\llbracket -2,1,-q_4,\dots,-q_{t+1}\rrbracket))\\
    &= f_{u,v}(\llbracket q_0-2,1\rrbracket\oplus g(\llbracket 2,-1,q_4,\dots,q_{t+1}\rrbracket))\\
    &= f_{u,v}(\llbracket q_0-2,1\rrbracket\oplus g(\llbracket 1\rrbracket\oplus -\llbracket q_4-1,q_5,\dots,q_{t+1}\rrbracket))\\
    &= f_{u,v}(\llbracket q_0-2,1\rrbracket\oplus g(\llbracket 1,1-q_4,-q_5,\dots,-q_{t+1}\rrbracket))\\
    &= f_{u,v}(\llbracket q_0-2,1\rrbracket\oplus(\llbracket0,1\rrbracket\oplus g(-\llbracket 2-q_4,-q_5,\dots,-q_{t+1}\rrbracket)))\\
    &= f_{u,v}(\llbracket q_0-2,2\rrbracket\oplus g(\llbracket q_4-2,q_5,\dots,q_{t+1}\rrbracket))\\
    &= f_{u,v}(\llbracket q_0-2,2,q^\diamond_0-2,q^\diamond_1,\dots,q^\diamond_z\rrbracket)\\
    &= \llbracket q_0,-1\rrbracket\oplus f_{u,v}(\llbracket 2,1,q^\diamond_0-2,q^\diamond_1,\dots,q^\diamond_z\rrbracket)\\
    &= \llbracket q_0,-1\rrbracket\oplus (\llbracket 4,-1\rrbracket\oplus f_{u,v}(\llbracket q^\diamond_0,q^\diamond_1,\dots,q^\diamond_z\rrbracket))\\
    &= \llbracket q_0,-1,4,-1\rrbracket\oplus (f_{u,v}\circ g)(\llbracket q_4,q_5,\dots,q_{t+1}\rrbracket))\\
    &= \llbracket q_0,-1,4,-1,q_4,q_5,\dots,q_{t+1}\rrbracket \text{ by the induction hypothesis.}
\end{align*}
So the result also holds in this case.

Having exhausted all possibilities, the result follows by strong induction. 
\end{proof}

Before presenting our main results, we restate a result from~\cite{HMST} that will be used several times throughout the subsequent proofs.

\begin{lemma}[\cite{HMST}, Lemma 2.10]\label{cf}
 Let $a/b$ be a rational number with $a/b=[q_0,q_1,\ldots,q_r]$, $\alpha\in\mathbb{Z}$, and $u$ and $v$ be nonnegative integers. It follows that
 \begin{enumerate}
    \item[(a)] $L_u^\alpha\begin{bmatrix}a\\b\end{bmatrix}=\begin{bmatrix}a\\au\alpha+b\end{bmatrix}$ and $R_v^\alpha\begin{bmatrix}a\\b\end{bmatrix}=\begin{bmatrix}a+bv\alpha\\b\end{bmatrix}$;
    \item[(b)] if $q_0=0$ and $u\alpha + q_1\neq 0$, then $\cfrac{a}{au\alpha+b}=[0,
u\alpha+ q_1,q_2,\dots,q_r]$;
    \item[(c)] if $q_0=0$ and $u\alpha + q_1=0$, then $\cfrac{a}{au\alpha+b}=[q_2,\dots,q_r]$;
    \item[(d)] if $q_0\neq0$, then $\cfrac{a}{au\alpha+b}=[0,
u\alpha,q_0,q_1,\dots,q_r]$; 
    \item[(e)]  $\cfrac{a+bv\alpha}{b}=[v\alpha+q_0,q_1,\dots,q_r]$.
 \end{enumerate}
\end{lemma}

The following proposition represents an extension of Proposition 3.2 in~\cite{HMST}. The proof that follows is essentially the proof of Proposition 3.1 in~\cite{HMST} modified accordingly and presented for completeness.

\begin{proposition}\label{oneFracGroup}
For integers $u,v\geq 1$, with $u+v>4$, and a matrix $M=\begin{bmatrix}
    a & b \\
    c & d
    \end{bmatrix}\in \mathscr{G}_{u,v}$, $(f_{v,u}\circ C)(c/a)$ satisfies the $(v,u)$-divisibility property if and only if $(f_{u,v}\circ C)(b/d)$ satisfies the $(u,v)$-divisibility property.
\end{proposition}

\begin{proof}
$(\Rightarrow)$ The proof in this direction is similar to that of the reverse direction. Since the main result is stated in terms of $b/d$, we omit this proof.

$(\Leftarrow)$ Suppose that $(f_{u,v}\circ C)(b/d)$ has the $(u,v)$-divisibility property. We first show that it must be the case that $(f_{u,v}\circ C)(b/d)=\llbracket v\alpha_0,u\alpha_1,\dots,v\alpha_{r-1}\rrbracket$ for some odd $r$.

Suppose that $(f_{u,v}\circ C)(b/d)=\llbracket v\alpha_0,u\alpha_1,\dots,u\alpha_r\rrbracket$ and consider the matrix $$N=R_v^{-\alpha_{r-1}}\cdots L_u^{-\alpha_1}R_v^{-\alpha_0}M.$$ By repeatedly applying Lemma~\ref{cf} parts (c) and (e), we get that $N=\begin{bmatrix}
    a' & b' \\
    c' & d'
    \end{bmatrix}$ with $b'/d'=1/(u\alpha_r),$ a contradiction since $N\in\mathscr{G}_{u,v}$.

So we must have that $(f_{u,v}\circ C)(b/d)=\llbracket v\alpha_0,u\alpha_1,\dots,v\alpha_{r-1}\rrbracket$ and, using the same definition for $N$ as above, $N=\begin{bmatrix}
    a' & 0 \\
    c' & d'
    \end{bmatrix}$. Since $N\in\mathscr{G}_{u.v}$, we must have that $a'=1$, $d'=1$ and $c'=u\alpha$ for some $\alpha\in\mathbb{Z},$ i.e., $N=L_u^\alpha$. Note that we cannot have $a'=-1$ and $d'=-1$, as these values are not congruent to $1\pmod{uv}$ as is required by the definition of $\mathscr{G}_{u,v}$. In particular, $M=R_v^{\alpha_0}L_u^{\alpha_1}\cdots R_v^{\alpha_{r-1}}L_u^\alpha.$ 

By repeatedly applying Lemma~\ref{cf} parts (d) and (e), we get that $c/a=[u\beta_0,v\beta_1,\dots,u\beta_k]$ for some $k\geq 0$, where $\beta_0\geq 0$ and $\beta_i>0$
for $0<i\leq k$. The result now follows from the fact that $C(c/a)=\llbracket u\beta_0,v\beta_1,\dots,u\beta_k\rrbracket$ by the definition of $C.$
\end{proof}

The following theorem represents our extension of Theorem 3.4 in~\cite{HMST}, where the proof has been modified accordingly and presented for completeness.

\begin{theorem}\label{sanovlikeGp}
For integers $u,v\geq 1$, with $u+v>4$, and a matrix $M=\begin{bmatrix}
    a & b \\
    c & d
    \end{bmatrix}\in \mathscr{G}_{u,v}$, $M\in G_{u,v}$ if and only if $\left(f_{u,v}\circ C\right)(b/d)$ satisfies the $(u,v)$-divisibility property. 
\end{theorem}

\begin{proof}
$(\Rightarrow)$ Suppose that $M\in G_{u,v}$. Then $M=R_v^{\alpha_0}L_u^{\alpha_1}R_v^{\alpha_2}\cdots R_v^{\alpha_{r-1}}L_u^{\alpha_r}$ where $\alpha_i\in\mathbb{Z}_{\neq 0}$ for $0<i<r$, $\alpha_0,\alpha_r\in\mathbb{Z}$, and $r$ is odd. For $0
\leq i\leq r$, let $$\begin{bmatrix}
    a_i & b_i \\
    c_i & d_i
    \end{bmatrix}=\begin{cases}
    L_u^{\alpha_{r-i}}\cdots R_v^{\alpha_{r-1}}L_u^{\alpha_r} & \text{ if $i$ is even},\\
    R_v^{\alpha_{r-i}}\cdots R_v^{\alpha_{r-1}}L_u^{\alpha_r} & \text{ if $i$ is odd}.
    \end{cases}$$
Then $\begin{bmatrix}b_0\\d_0\end{bmatrix}=\begin{bmatrix}0\\1\end{bmatrix}$ and, by Lemma~\ref{cf} part (a), for $0<i\leq r$, $$\begin{bmatrix}b_i\\d_i\end{bmatrix} = \begin{cases}
\begin{bmatrix}b_{i-1}\\b_{i-1}u\alpha_{r-i}+d_{i-1}\end{bmatrix} & \text{ if $i$ is even},
\\[3ex]
\begin{bmatrix}b_{i-1}+d_{i-1}v\alpha_{r-i}\\d_{i-1}\end{bmatrix} &\text{ if $i$ is odd}.
\end{cases}$$
By repeatedly applying Lemma~\ref{cf} parts (d) and (e) to $b_i/d_i$, it follows that $b/d=b_r/d_r=[v\alpha_0,u\alpha_1,\dots,v\alpha_{r-1}].$ To complete this direction of the proof we must show that $(f_{u,v}\circ C)(b/d)=\llbracket v\alpha_0,u\alpha_1,\dots,v\alpha_{r-1}\rrbracket.$ We can assume that $r\geq 3$ since the result is trivial when $r=1.$

Since $\llbracket v\alpha_0,u\alpha_1,\dots,v\alpha_{r-1}\rrbracket\in A_2$, then by Lemma~\ref{guvE}, $g(\llbracket v\alpha_0,u\alpha_1,\dots,v\alpha_{r-1}\rrbracket)=C(b/d)$. By Proposition~\ref{fuvg}, we see that
\begin{align*}
(f_{u,v}\circ C)\left(\frac{b}{d}\right) &= (f_{u,v}\circ g)(\llbracket v\alpha_0,u\alpha_1,\dots,v\alpha_{r-1}\rrbracket)\\
&= \llbracket v\alpha_0,u\alpha_1,\dots,v\alpha_{r-1}\rrbracket.
\end{align*}

$(\Leftarrow)$ In this case we obtain the proof by following the same argument given in the first paragraph of the `if' argument in Proposition 3.1 in~\ref{oneFracGroup}.
\end{proof}

A careful reading of Theorem~\ref{sanovlikeGp} shows that our method, as in our previous version, allows one to determine the exponents in the alternating product representation of $M$, should it be the case that $M\in G_{u,v}$. Furthermore, the result given by Theorem~\ref{sanovlikeGp} can be extended to negative values for $u$ and $v$ since $L_{-u}=L_{u}^{-1}$ and $R_{-v}=R_{v}^{-1}$.


\subsection{Index Results}

In this section we compute the value of $[\mathscr{G}_{u,v}\colon G_{u,v}]$ for all $u,v\geq 1$. Before doing so we present some results on the relationship between $\mathscr{G}_{u,v}$ and $G_{u,v}$ when $u+v\leq 4$. 

\begin{lemma}\label{negig}
We have that $-I_2\in G_{u,v}$ if and only if $u+v\leq 3.$
\end{lemma}

\begin{proof}
($\Rightarrow$) Suppose that $-I_2\in G_{u,v}.$ Then $-I_2\in\mathscr{G}_{u,v}$. By the definition of $\mathscr{G}_{u,v}$ there is an integer $n$ such that $1+uvn=-1$. For this to occur we would need $uvn=-2$. It follows that $u+v\leq 3.$

($\Leftarrow$) This direction follows from the fact that $-I_2=(R_2L_1^{-1})^2=(L_2R_1^{-1})^2.$ 
\end{proof}

The following propositions show that we get a Sanov-like result when
$u+v=3$ or $4$.

\begin{proposition}\label{1-2Case}
If $M\in\mathscr{G}_{1,2}$, then $M\in G_{1,2}$.
\end{proposition}

\begin{proof}
Let $M=\begin{bmatrix}
    a & b \\
    c & d
    \end{bmatrix}\in \mathscr{G}_{1,2}$ and let $[q_0,q_1,\dots,q_r]$ be the short continued fraction representation of $b/d.$ Following the ideas in the proof of Theorem~\ref{sanovlikeGp}, we must show that this continued fraction representation can be manipulated to an equivalent form so that all coefficients with an even index are even.

    If $r=0$, then  the only possibilities for $d$ are $\pm 1$. When  $d=1$, we have $M=\begin{bmatrix}
    1+2n_1 & 2n_2 \\
    n_3 & 1
    \end{bmatrix}$ with $n_1=n_2n_3$.  In this case, $M=R_2^{n_2}L_1^{n_3}$, so $M\in G_{1,2}$. If $d=-1$, we apply the previous argument to $-M$, which gives that $M=-R_2^{n_2}L_1^{n_3}\in G_{1,2}$. 
    
    Suppose $r>0$. By repeatedly applying Lemma~\ref{secondid} from left to right with $\alpha$ representing the leftmost odd partial quotient in the continued fraction representation of $b/d,$ we obtain a continued fraction representation $$\frac{b}{d}=[\alpha_0,\alpha_1,\dots,\alpha_{s-1},\alpha_s],$$ where $\alpha_i$ is even whenever $i\equiv 0\bmod 2$ and $0\leq i \leq s-2$.  
If $s$ is odd,
$$[\alpha_{s-1}, \alpha_{s}]    
=\begin{cases}
[\alpha_{s-1}, \alpha_{s}] &\text{ if $\alpha_{s-1}$ is even},\\
[\alpha_{s-1}+1,-1,-(\alpha_{s}-1)] &\text{ if $\alpha_{s-1}$ is odd and $\alpha_{s}$ is odd},\\
[\alpha_{s-1}+1,-1,-\alpha_{s},1] &\text{ if $\alpha_{s-1}$ is odd and $\alpha_{s}$ is even}.\\
\end{cases}$$
If $s$ is even,
$$[\alpha_{s-1}, \alpha_{s}]    
=\begin{cases}
[\alpha_{s-1}, \alpha_{s}] &\text{ if $\alpha_{s}$ is even},\\
[\alpha_{s-1}+1] &\text{ if $\alpha_{s}=1$},\\
[\alpha_{s-1}, \alpha_{s}-1,1] &\text{ if $\alpha_{s}$ is odd and $\neq 1$}.\\
\end{cases}$$
The cases follow directly from the short and long continued fraction representations $[\alpha, \beta]=[\alpha, \beta-1, 1]$ and Lemma~\ref{firstid}(b). 
\end{proof}

\begin{proposition}\label{1-3Case}
If $M\in\mathscr{G}_{1,3}$, then $M\in G_{1,3}$.
\end{proposition}

\begin{proof}
Let $M=\begin{bmatrix}
    a & b \\
    c & d
    \end{bmatrix}\in \mathscr{G}_{1,3}$ and let $[q_0,q_1,\dots,q_r]$ be the short continued fraction representation of $b/d.$ As in the previous proof, we must show that this continued fraction representation can be manipulated to an equivalent form so that all coefficients with an even index are divisible by $3$. We will prove the desired result by strong induction on r.
    
    If $r=0$, then  it must be the case that $d=1$. We therefore have $M=\begin{bmatrix}
    1+3n_1 & 3n_2 \\
    n_3 & 1
    \end{bmatrix}$ with $n_1=n_2n_3$.  In this case, $M=R_3^{n_2}L_1^{n_3}$, so $M\in G_{1,3}$. If $r=1$, then $b/d=[q_0,q_1]$ with $q_1\neq\pm 1$ and $q_0\equiv \pm 1\bmod{3}$. Suppose that $q_0\equiv 1\bmod{3}$. Then $q_0=3\alpha_0+1$ for some integer $\alpha_0$ and it follows that
    \begin{align*}
        [q_0,q_1] &= [3\alpha_0,1,-(q_1+1)]\\
        &= \begin{cases}
            [3\alpha_0,1,-q_1,-1]&\text{ if $q_1\equiv 0\bmod{3}$,}\\
            [3\alpha_0,1,-(q_1+2),1]&\text{ if $q_1\equiv 1\bmod{3}$,}\\
            [3\alpha_0,1,-(q_1+1)]&\text{ if $q_1\equiv -1\bmod{3}$.}
        \end{cases}
    \end{align*}
A similar argument works if $q_0\equiv -1 \bmod{3}.$

Suppose that the result holds for $0\leq r< t$ for some $t\geq 2$. Then $b/d=[q_0,q_1,\dots,q_r]$ with $r\geq 2$. If $q_0\equiv 0\bmod{3}$, then the problem reduces to the case $[q_2,\dots,q_r]$ for the matrix $L_1^{-q_1}R_3^{-q_0/3}M$. So we can assume that $q_0\equiv \pm 1\bmod{3}.$ Suppose that $q_0\equiv 1\bmod{3}$. Then $q_0=3\alpha_0+1$ for some integer $\alpha_0$ and by Lemma~\ref{firstid}(a), we have
\begin{align*}
    b/d &=[q_0,q_1,\dots,q_r]\\
    &=[3\alpha_0+1,q_1,\dots,q_r]\\
    &=[3\alpha_0,1,-(q_1+1),-q_2,\dots,-q_r].
\end{align*}
As in the $r=1$ base case, this reduces the problem to the matrix $L_1^{-1}R_3^{-\alpha_0}M$, which holds by induction. A similar argument works if $q_0\equiv -1 \bmod{3}.$ So the result holds by induction.
\end{proof}

Propositions~\ref{1-2Case} and~\ref{1-3Case} together with Sanov's result show that $G_{u,v}=\mathscr{G}_{u,v}$ when $u+v\leq 4$. The theorem below shows that there are no other cases where this is true.

\begin{theorem}\label{Guv.equal.sGuv}
We have that $G_{u,v}=\mathscr{G}_{u,v}$ if and only if $u+v\leq 4$. 
\end{theorem}

\begin{proof}
($\Rightarrow$) Suppose $u+v>4$. Assume $v\geq u\geq 2$. Let $M_{u,v}(k)=\begin{bmatrix}
    1-uvk & uvk \\
    -uvk & 1+uvk
    \end{bmatrix}$ where $k$ is a positive integer. Then $M_{u,v}(k)\in \mathscr{G}_{u,v}$ with
\begin{align*}
    (f_{u,v}\circ C)\left(\frac{uvk}{1+uvk}\right) &= f_{u,v}(\llbracket 0,1,uvk\rrbracket)\\
    &= \llbracket 0,1,uvk\rrbracket.
\end{align*}
By Theorem~\ref{sanovlikeGp}, $M_{u,v}(k)\not\in G_{u,v}$. Now assume $v\geq 4$ and $u=1$. Let $N_v(k)=\begin{bmatrix}
    1+vk & -vk^2 \\
    v & 1-vk
    \end{bmatrix}$ where $k>2$ is a positive integer with $k\equiv 2\bmod{v}$. Then $N_v(k)\in \mathscr{G}_{1,v}$ with
\begin{align*}
    (f_{1,v}\circ C)\left(\frac{-vk^2}{1-vk}\right) &= 
        f_{1,v}(\llbracket k,v-1,1,k-1\rrbracket)\\
    &= \llbracket k\rrbracket\oplus f_{v,1}(\llbracket v-1,1,k-1\rrbracket).
\end{align*} 
By Theorem~\ref{sanovlikeGp}, $N_v(k)\not\in G_{u,v}$.

We therefore obtain the desired result by contraposition.

($\Leftarrow$) Sanov's result together with Propositions~\ref{1-2Case} and~\ref{1-3Case} immediately give the desired result.
\end{proof}

\begin{theorem}
We have that  
\begin{align*}
[\mathscr{G}_{u,v}\colon G_{u,v}]=
\begin{cases}
1 & u+v\leq 4,\\
\infty &\text{otherwise.}
\end{cases}
\end{align*}
\end{theorem}

\begin{proof}
From Theorem~\ref{Guv.equal.sGuv} we get the desired result when $u+v\leq 4$, so we may assume otherwise.

For $v\geq u\geq 2$ we use the fact that $M_{u,v}(k)=\left(M_{u,v}(1)\right)^k$. Suppose that $k_1,k_2\in\mathbb{Z}$ are such that $M_{u,v}(k_1)$ and $M_{u,v}(k_2)$ are representatives of the same coset. That is, assume
\begin{equation*}
    M_{u,v}(k_1)G_{u,v} = M_{u,v}(k_2)G_{u,v}.
\end{equation*}
Without loss of generality, suppose $k_1\geq k_2$. The above equality occurs if and only if $M_{u,v}(k_1)(M_{u,v}(k_2))^{-1}\in G_{u,v}$. But
\begin{align*}
M_{u,v}(k_1)(M_{u,v}(k_2))^{-1} &= (M_{u,v}(1))^{k_1}(M_{u,v}(1))^{-k_2}\\
&= (M_{u,v}(1))^{k_1-k_2}\\
&= M_{u,v}(k_1-k_2).
\end{align*}
For $M_{u,v}(k_1-k_2)\in G_{u,v}$ to hold, it must be that $k_1=k_2$. This means that $\{M_{u,v}(k)G_{u,v}:k\geq 0\}$ represents an infinite family of distinct left cosets.

For $v\geq 4$ and $u=1$, we use the fact that if $N_v(k_1)\left(N_v(k_2)\right)^{-1}=\begin{bmatrix}
    a & b \\
    c & d
    \end{bmatrix}$, then $c/a=[0,k_1,v,k_2-k_1,-v]$ where $k_1,k_2\equiv 2\bmod{v}$. Suppose that $k_1,k_2\in\mathbb{Z}$ are such that $N_v(k_1)$ and $N_v(k_2)$ are representatives of the same coset. That is, assume
\begin{equation*}
    N_v(k_1)G_{u,v} = N_v(k_2)G_{u,v}.
\end{equation*}
Without loss of generality, suppose $k_1\leq k_2$. Based on the computations in the proof of Theorem~\ref{Guv.equal.sGuv}, the above equality occurs if and only if $N_v(k_1)\left(N_v(k_2)\right)^{-1}\in G_{u,v}$, in contradiction to Theorem~\ref{sanovlikeGp}.
This means that $\{N_v(k)G_{u,v}:k\geq 0,k\equiv2\bmod{v}\}$ represents an infinite family of distinct left cosets.
\end{proof}

\section{Example}

The following example shows our motivation for defining $f_{u,v}$ in order to extend our previous results.

Consider the matrix $M=\begin{bmatrix}2401 & 12975 \\ 250 & 1351 \end{bmatrix}=R_3^{3}L_2R_3^{-1}L_2^{5}R_3L_2^{-1}R_3^{2}.$ Applying $f\circ C$ to $\frac{12975}{1351}$ as was done in~\cite{HMST} gives that
\begin{align*}
    (f\circ C)\left(\frac{12975}{1351}\right) 
    &= \llbracket10,-3,2,9,2,2,5\rrbracket,
\end{align*}
which clearly does not satisfy the $(2,3)$-divisibility property despite the fact that $M\in G_{2,3}.$ The main issue here is the way that the 1's were eliminated. Using $f_{2,3}$ in place of $f$ gives that
\begin{align*}
    (f_{2,3}\circ C)\left(\frac{12975}{1351}\right) &= f_{2,3}(\llbracket9,1,1,1,1,9,2,2,5\rrbracket)\\
    &= \llbracket9\rrbracket\oplus f_{3,2}(\llbracket1,1,1,1,9,2,2,5\rrbracket)\\
    &= \llbracket9\rrbracket\oplus(\llbracket2\rrbracket\oplus -f_{2,3}(\llbracket2,1,9,2,2,5\rrbracket))\\
    &= \llbracket9\rrbracket\oplus(\llbracket2\rrbracket\oplus-(\llbracket3\rrbracket\oplus-f_{3,2}(\llbracket10,2,2,5\rrbracket)))\\
    &= \llbracket9\rrbracket\oplus(\llbracket2\rrbracket\oplus-(\llbracket3\rrbracket\oplus-(\llbracket10\rrbracket\oplus f_{2,3}(\llbracket2,2,5\rrbracket))))\\
    &= \llbracket9\rrbracket\oplus(\llbracket2\rrbracket\oplus-(\llbracket3\rrbracket\oplus-(\llbracket10\rrbracket\oplus (\llbracket3\rrbracket\oplus f_{3,2}(\llbracket-2,6\rrbracket)))))\\
    &= \llbracket9\rrbracket\oplus(\llbracket2\rrbracket\oplus-(\llbracket3\rrbracket\oplus-(\llbracket10\rrbracket\oplus (\llbracket3,-2\rrbracket\oplus f_{2,3}(\llbracket6\rrbracket)))))\\
    &= \llbracket9\rrbracket\oplus(\llbracket2\rrbracket\oplus-(\llbracket3\rrbracket\oplus-(\llbracket10\rrbracket\oplus (\llbracket3,-2\rrbracket\oplus \llbracket6\rrbracket))))\\
    &= \llbracket9,2,-3,10,3,-2,6\rrbracket,
\end{align*}
which does satisfy the $(2,3)$-divisibility property, as desired, and encodes the exponents in the product representation of $M$.

\section{Concluding Remarks}

The definition of $f_{u,v}$ shows that the function always changes the value of the partial quotients by no more than 2 in absolute value. In effect, the function $f_{u,v}$ takes in a vector and attempts to output an equivalent vector satisfying the $(u,v)$-divisibility property by adjusting the components according to Lemmas~\ref{firstid} and~\ref{secondid}. Two questions may come to mind at this stage: 
\begin{itemize}
    \item Why are Lemmas~\ref{firstid} and~\ref{secondid} enough for our purposes?
    \item Is there some version of Lemmas~\ref{firstid} and~\ref{secondid} that can adjust the components by more?
\end{itemize}
Both questions can be addressed by noting that 
\begin{equation}\label{kadjust}
    [\alpha,\beta,\gamma]=[\alpha+k,1,-1,k+1,-\beta,-\gamma]
\end{equation} 
for $k\in\mathbb{Z}.$ This identity shows that attempting to adjust a component by a `large' value of $k$ may introduce new components that need to be adjusted as well. In some cases, this is problematic. For example, consider $\llbracket2,3\rrbracket$ and suppose that we are attempting to find some equivalent vector that satisfies the $(8,1)$-divisibility property. Using~\eqref{kadjust}, we see that we can indeed raise the 2 to a multiple of 8, giving us the equivalent vector $\llbracket8,1,-1,7,-3\rrbracket.$ Unfortunately this creates an endless set of problems each time that we attempt to fix every other component. 

By applying Lemma~\ref{firstid}(b) repeatedly, we get 

\begin{align*}
[\alpha,\beta,\gamma] &= [\alpha+k,-1, \underbrace{2,-2,2,\dots,\pm 2}_{k-1\text{ terms}},(-1)^{k}(\beta-1), (-1)^{k}\gamma].
\end{align*}
When $u=1$ and $v = 4$, this continued fraction sequence is already ``beyond repair" in the sense that we are unable to make $[2,-2,2,\dots, \pm2]$ satisfy $(1,4)$-divisibility within a finite number of steps. By applying Lemma~\ref{firstid}(b) twice, we have 
$$[2, -2] \longrightarrow [3, -1, 2] \longrightarrow [4, -1, 2, -2].$$
The reappearance of $[2, -2]$ in the ``corrected" sequence shows that the $(1,4)$-divisibility can never be reached.


\end{document}